\documentclass[12pt,oneside]{amsart}
\usepackage{latexsym,amssymb,amsmath}

\numberwithin{equation}{section}
\newtheorem{theorem}{Theorem}[section]
\newtheorem{lemma}[theorem]{Lemma}

\theoremstyle{definition}
\newtheorem{definition}[theorem]{Definition}
\newtheorem{remark}[theorem]{Remark}
\newtheorem{example}[theorem]{Example}
\newtheorem{application}[theorem]{Application}

\newcommand{\pnpm}{\mathbb{P}^n \times \mathbb{P}^m}

\newcommand{\Iz}{I_{Z}}

\newcommand{\pr}{\mathbb{P}}
\newcommand{\supp}{\operatorname{Supp}}

\newcommand{\popo}{\mathbb{P}^1\times \mathbb{P}^1 }

\newcommand{\N}{\mathbb{N}}

\newcommand{\Ssz}{\mathcal S_{Z}}

\title{Separators of Arithmetically Cohen-Macaulay fat points in $\popo$}
\thanks{The second author acknowledges the support of NSERC; Updated: May 28, 2010}

\begin{document}
\author{Elena Guardo}
\address{Dipartimento di Matematica e Informatica\\
Viale A. Doria, 6 - 95100 - Catania, Italy}
\email{guardo@dmi.unict.it}

\author{Adam Van Tuyl}
\address{Department of Mathematical Sciences\\
Lakehead University, Thunder Bay, ON, P7B 5E1, Canada}
\email{avantuyl@lakeheadu.ca}


\begin{abstract}
Let $Z \subseteq \popo$ be a set of fat points that
is also arithmetically Cohen-Macaulay (ACM).  We
describe how to compute the degree of a separator
of a fat point of multiplicity $m$ for each point
in the support of $Z$ using only a numerical
description of $Z$.  Our formula  extends
the case of reduced points which was previously known.
\end{abstract}
\maketitle


\section{Introduction}

Fix an algebraically closed field $k$ of characteristic zero.
Let $R = k[x_0,x_1,y_0,y_1]$ be the $\N^2$-graded polynomial ring
with $\deg x_i = (1,0)$ for $i=0,1$ and $\deg y_i = (0,1)$ for
$i=0,1$.  The ring $R$ is the coordinate ring of $\popo$. Consider now
a set of points $X = \{P_1,\ldots,P_s\}\subset \popo$, and fix
positive integers $m_1,\ldots,m_s$.  The goal of this note is to
study some of the properties of the scheme $Z = m_1P_1 + \cdots +
m_sP_s$ of fat points (precise definitions are deferred until the
next section).  In particular, we are interested in describing the
separator of $P_i$ of multiplicity $m_i$.

Recall that for sets of points $X = \{P_1,\ldots,P_s\} \subseteq
\pr^n$, a homogeneous form $F \in k[\pr^n]$ is called a {\bf
separator} of $P \in X$ if $F(P) \neq 0$, but $F(Q) = 0$ for all $Q \in
X \setminus \{P\}$.  Over the years, a number of authors
have shown how to exploit information about
the separator of a point to describe properties of the set
of reduced points $X \subseteq \pr^n$ (e.g., see
\cite{ABM,B,BC,GKR,K,O,So}).
In a series of papers, the authors, along with
Marino, (see \cite{GMVT,GVT3}) generalized
some of these results by studying separators
of fat points, a family of non-reduced points.  Roughly
speaking, a {\bf separator of a point $P_i$ of multiplicity $m_i$}
and the {\bf degree of a point $P_i$ of multiplicity $m_i$} are
defined in terms of the generators of $I_{Z'}/I_Z$ in $R/I_Z$
where $I_{Z'}$ is the defining ideal of  $Z' = m_1P_1 +
\cdots + (m_i-1)P_i + \cdots + m_sP_s$.

General properties of separators of both reduced points and fat points
in a multiprojective space  were studied in \cite{GVT1,GVT2,GVT3, M1}.
We now specialize to the case of fat points in $\popo$.  By
restricting to this case, we can improve upon the results found
in \cite{GVT3}.  The main result of this paper (Theorem \ref{degP1xP1}) is to show how
to compute the degree of a point $P_i$ of multiplicity
$m_i$ directly from the combinatorics of the scheme, i.e., the
number of points on the various rulings of $\popo$ and the multiplicities,
provided that the scheme is ACM.  This result generalizes
the reduced case as found in
\cite[Theorem 7.4]{M1} and \cite[Theorem 4.4]{GVT2}.

\noindent {\bf Acknowledgments.}  We thank  Brian Harbourne for
his useful comments.


\section{Preliminaries and notation}

\subsection{ACM fat points in $\popo$}
Let $X$ be a set of $s$ distinct points in $\popo$. Let
$\pi_1:\popo \rightarrow \pr^1$ denote the projection morphism
defined by $P \times Q \mapsto P$.  Let $\pi_2:\popo \rightarrow
\pr^1$ be the other projection morphism. The set $\pi_1(X) =
\{P_1,\ldots,P_a\}$ is the set of $a \leq s$  distinct first
coordinates that appear in $X$.  Similarly, $\pi_2(X) =
\{Q_1,\ldots,Q_b\}$ is the set of $b \leq s$ distinct second
coordinates.  For $i = 1,\ldots,a$, let $L_{P_i}$ be the degree
$(1,0)$ form that vanishes at all the points with
first coordinate $P_i$. Similarly, for $j = 1,\ldots,b$, let
$L_{Q_j}$ denote the degree $(0,1)$ form that vanishes at points
with second coordinate $Q_j$.

Let $D:=\{(x,y) ~|~ 1 \leq x \leq a, 1 \leq y \leq b\}.$ If
$P \in X$, then $I_{P} = (L_{R_i},L_{Q_j})$ for some $(i,j) \in D.$ So,
we can write each point $P \in X$ as $P_i \times Q_j$ for
some $(i,j)\in D$. (Note: this does not mean that if $(i,j)
\in D$, then  $P_i \times Q_j \in X$;  there may exist a tuple $(i,j) \in D$,
but $P_i \times Q_j \not\in X$.)

Suppose that $X$ is a set of distinct points in $\popo$ and
$|\pi_1(X)| = a$ and $|\pi_2(X)| = b$. Let $I_{P_i \times Q_j} =
(L_{R_i},L_{Q_j})$ denote the ideal associated to the point
$P_i \times Q_j \in X$. For each $(i,j) \in D$, let $m_{ij}$ be a positive
integer if $P_i \times Q_j \in X$, otherwise, let $m_{ij} = 0$. Then we
denote by $Z$ the subscheme of $\popo$ defined by the saturated
bihomogeneous ideal
\[
\Iz = \bigcap_{(i,j) \in D} I_{P_i\times Q_j}^{m_{ij}}
\]
where $I_{P_i \times Q_j}^0 := (1)$. We say $Z$ is a {\bf fat point scheme}
or {\bf a set of fat points}
of $\popo$. The integer $m_{ij}$ is called the {\bf multiplicity} of
the point $P_i \times Q_j$. We shall sometimes denote the scheme
as $Z = \{ (P_i \times Q_j;m_{ij}) ~|~ (i,j) \in D\}$, or as
$Z = m_{11}(P_1 \times Q_1) + \cdots + m_{ab}(P_a \times Q_b)$.
The {\bf support} of $Z$,
written $\supp(Z)$, is the set of points $X$.

A fat point scheme is said to be {\bf arithmetically
Cohen-Macaulay} (ACM for short) if the associated coordinate ring
is Cohen-Macaulay.  We will need a classification of
ACM fat point schemes of $Z \subseteq \popo$ due to the first author
(see \cite{Gu2}).

We begin by recalling a construction of \cite{Gu2}.
Let $Z$ be a fat point scheme in $\popo$ where $Z =
\{(P_i \times Q_j ;m_{ij}) ~|~ 1 \leq i \leq a, ~1 \leq j \leq b\}$ with
$m_{ij} \geq 0$. For each $h \in \N$, and for each tuple $(i,j)$ with
$1 \leq i \leq a$ and $1 \leq j \leq b$, define
\[
t_{ij}(h) := (m_{ij}-h)_+ = \max\{0,m_{ij}-h\}.
\]
The set $\Ssz$ is then defined to be the set of $b$-tuples
\[
\Ssz=\{(t_{11}(h),\ldots,t_{1b}(h)),(t_{21}(h),\ldots,t_{2b}(h)),\ldots,
(t_{a1}(h),\dots,t_{ab}(h)) ~|~ h \in \N\}.
\]
The elements of $\Ssz$ belong to $\N^b$.  Let $\succeq$ denote
the partial order where $(i_1,\ldots,i_b) \succeq (j_1,\ldots,j_b)$
if and only if $i_\ell \geq j_\ell$ for all $\ell = 1,\ldots,b$.
Then we have \cite[Theorem 4.8]{Gu2}:

\begin{theorem} \label{equivalent}
Let $Z$ be a fat point scheme in $\popo$.  Then $Z$ is ACM
if and only if the elements of $\Ssz$ can be totally
ordered by $\succeq$, i.e., $\Ssz$ has no incomparable
elements.
\end{theorem}

\begin{remark}Recall that the bigraded {\bf Hilbert function}
of $Z$ is defined by
$H_Z(i,j) = \dim_k R_{i,j}-\dim_k (I_Z)_{i,j}$.
When $Z \subseteq \popo$ is an ACM fat point scheme,
then $H_Z(i,j)$ can be computed for all $(i,j)$ directly
from the set $\Ssz$ (see \cite{Gu2} for details). 
\end{remark}
\begin{example} \label{mainexample}
We illustrate these ideas with
the following set of fat points in $\popo$:
\begin{center}
\begin{picture}(150,110)(25,-10)
\put(0,40){$Z = $}\put(60,-10){\line(0,1){90}}
\put(80,-10){\line(0,1){90}} \put(100,-10){\line(0,1){90}}
\put(120,-10){\line(0,1){90}} \put(58,95){$Q_1$}
\put(78,95){$Q_2$} \put(98,95){$Q_3$} \put(118,95){$Q_4$}

\put(64,77){$5$}
\put(84,77){$4$}
\put(104,77){$2$}
\put(124,77){$2$}

\put(64,57){$5$}
\put(84,57){$3$}
\put(104,57){$2$}
\put(124,57){$1$}

\put(64,37){$4$}
\put(84,37){$3$}
\put(104,37){$1$}

\put(64,17){$2$}
\put(84,17){$1$}

\put(64,-2){$1$}

\put(55,-5){\line(1,0){70}}
\put(55,15){\line(1,0){70}}
\put(55,35){\line(1,0){70}}
\put(55,55){\line(1,0){70}}
\put(55,75){\line(1,0){70}}
\put(37,73){$P_1$}
\put(37,53){$P_2$}
\put(37,33){$P_3$}
\put(37,13){$P_4$}
\put(37,-7){$P_5$}

\put(60,75){\circle*{5}}
\put(80,75){\circle*{5}}
\put(100,75){\circle*{5}}
\put(120,75){\circle*{5}}

\put(60,55){\circle*{5}}
\put(80,55){\circle*{5}}
\put(100,55){\circle*{5}}
\put(120,55){\circle*{5}}

\put(60,35){\circle*{5}}
\put(80,35){\circle*{5}}
\put(100,35){\circle*{5}}

\put(60,15){\circle*{5}}
\put(80,15){\circle*{5}}
\put(60,-5){\circle*{5}}
\end{picture}
\end{center}
Consider the multiplicities $5,4,2,$ and $2$ on the first ruling,
i.e., the points whose first coordinate is $P_1$. From the construction of $\Ssz$, the tuples
$(5,4,2,2)$,$(4,3,1,1)$,$(3,2,0,0)$, $(2,1,0,0)$, $(1,0,0,0)$ all belong
to $\Ssz$.  Notice we successively subtract one from each entry,
until we reach a zero.  We carry out this procedure for each
ruling to find:
\begin{eqnarray*}
\Ssz &=& \{(5,4,2,2),(4,3,1,1),(3,2,0,0),(2,1,0,0),(1,0,0,0),(5,3,2,1), \\
&&        (4,2,1,0),(3,1,0,0), (2,0,0,0),(1,0,0,0),(4,3,1,0),(3,2,0,0),\\
&&(2,1,0,0),(1,0,0,0),(2,1,0,0),(1,0,0,0),(1,0,0,0)\}.
\end{eqnarray*}
The partial order $\succeq$, when restricted to the set $\Ssz$,
is a total ordering, i.e., there are no incomparable elements.
Thus, the set $Z$ is an ACM set of fat points.
\end{example}

\begin{remark}\label{order} We note that that the following two schemes
of fat points in $\popo$:
\begin{center}
\begin{picture}(150,120)(75,-10)
\put(0,40){$Z = $} \put(135,40){$Z' =$}
\put(60,-10){\line(0,1){90}}
\put(190,-10){\line(0,1){90}}\put(80,-10){\line(0,1){90}}\put(210,-10){\line(0,1){90}}
\put(100,-10){\line(0,1){90}}\put(230,-10){\line(0,1){90}}
\put(120,-10){\line(0,1){90}} \put(250,-10){\line(0,1){90}}
\put(58,95){$Q_1$} \put(78,95){$Q_2$} \put(98,95){$Q_3$}
\put(118,95){$Q_4$}

\put(188,95){$Q'_1$} \put(208,95){$Q'_2$} \put(228,95){$Q'_3$}
\put(248,95){$Q'_4$}

\put(64,77){$5$}
\put(194,77){$1$}\put(84,77){$4$}\put(214,77){$2$}
\put(104,77){$2$}\put(234,77){$0$}
\put(124,77){$2$}\put(254,77){$0$}

\put(64,57){$5$} \put(84,57){$3$} \put(104,57){$2$}
\put(124,57){$1$}

\put(194,57){$3$} \put(214,57){$4$} \put(234,57){$1$}
\put(254,57){$0$}

\put(64,37){$4$} \put(84,37){$3$} \put(104,37){$1$}

\put(194,37){$4$} \put(214,37){$5$} \put(234,37){$2$}
\put(254,37){$2$}

\put(64,17){$2$} \put(84,17){$1$} \put(194,17){$3$}
\put(214,17){$5$} \put(234,17){$2$} \put(254,17){$1$}

\put(64,-2){$1$} \put(194,-2){$0$} \put(214,-2){$1$}
\put(234,-2){$0$} \put(254,-2){$0$}

\put(55,-5){\line(1,0){70}} \put(55,15){\line(1,0){70}}
\put(55,35){\line(1,0){70}} \put(55,55){\line(1,0){70}}
\put(55,75){\line(1,0){70}} \put(37,73){$P_1$} \put(37,53){$P_2$}
\put(37,33){$P_3$} \put(37,13){$P_4$} \put(37,-7){$P_5$}

\put(167,73){$P'_1$} \put(167,53){$P'_2$}
\put(167,33){$P'_3$} \put(167,13){$P'_4$} \put(167,-7){$P'_5$}

\put(185,-5){\line(1,0){70}} \put(185,15){\line(1,0){70}}
\put(185,35){\line(1,0){70}} \put(185,55){\line(1,0){70}}
\put(185,75){\line(1,0){70}}

\put(60,75){\circle*{5}} \put(80,75){\circle*{5}}
\put(100,75){\circle*{5}} \put(120,75){\circle*{5}}

\put(190,75){\circle*{5}} \put(210,75){\circle*{5}}

\put(60,55){\circle*{5}} \put(80,55){\circle*{5}}
\put(100,55){\circle*{5}} \put(120,55){\circle*{5}}

\put(190,55){\circle*{5}} \put(210,55){\circle*{5}}
\put(230,55){\circle*{5}}

\put(60,35){\circle*{5}} \put(80,35){\circle*{5}}
\put(100,35){\circle*{5}}

\put(190,35){\circle*{5}} \put(210,35){\circle*{5}}
\put(230,35){\circle*{5}} \put(250,35){\circle*{5}}

\put(60,15){\circle*{5}} \put(80,15){\circle*{5}}
\put(190,15){\circle*{5}} \put(210,15){\circle*{5}}
\put(230,15){\circle*{5}} \put(250,15){\circle*{5}}

\put(60,-5){\circle*{5}} \put(210,-5){\circle*{5}}

\end{picture}
\end{center}
\noindent are both ACM, and reordering in a suitable way the lines
of type $(1,0)$ and $(0,1)$ they become the same. Thus, if $Z$ is
ACM, we can always suppose that $(m_{i1},\dots, m_{ib})\succeq
(m_{j1},\dots, m_{jb})$ for  $1 \leq i < j \leq a$.
\end{remark}

When $Z$ is ACM, there are some relative bounds on the
multiplicities:

\begin{lemma}\label{bounds}
Let $Z \subseteq \popo$ be an ACM set of fat points.
\begin{enumerate}
\item[$(i)$]
Suppose that there exists $i,k,j,l$ such that $P_i \times Q_j, P_i
\times Q_l, P_k \times Q_j$, and $P_k \times Q_l$ all belong to
$\supp(Z)$ and let $m_{ij},m_{il}$,$m_{kj}$, and $m_{kl}$ be
the corresponding nonzero multiplicities. Then
$m_{ij} \leq m_{il} + m_{kj} - m_{kl} + 1.$
\item[$(ii)$] Suppose that there
exists $i,k,j,l$ such that $P_i \times Q_j, P_i \times Q_l,$ and
$P_k \times Q_j$, all belong to $\supp(Z)$, but $P_k \times Q_l
\not\in \supp(Z)$ and let $m_{ij},m_{il}$,and $m_{kj}$ be the
corresponding nonzero multiplicities. Then
$m_{il} \leq m_{ij} - m_{kl} + 1.$
\end{enumerate}
\end{lemma}

\begin{proof}
$(i)$ Since $Z$ is ACM, using Remark \ref{order} we can always
reorder the lines of type $L_{P_i}$ and $L_{Q_j}$ so that
$m_{ij}\geq m_{il}$ and $m_{ij}\geq m_{kj}.$ Suppose that
$m_{ij} > m_{il} + m_{kj} - m_{kl} + 1.$ From the construction
of $\Ssz$, the tuples
\[(\star,m_{ij} - m_{il} +  m_{kl} -1, \star, m_{il} - m_{il} + m_{kl}-1,\star)
~~\text{and}~~ (\star,m_{kj},\star,m_{kl},\star)\] are elements
of $\Ssz$ where $\star$ denotes other elements of the tuple. But since
$m_{ij} - m_{il} +  m_{kl} -1 > m_{kj}$ but $0 \leq m_{il} -
m_{il} + m_{kl}-1 = m_{kl} - 1 < m_{kl}$, these tuples of $\Ssz$
will be incomparable, which contradicts the ACM property of $Z$.

$(ii)$  The proof is similar to $(i)$.  Suppose that $m_{il} >
m_{ij} - m_{kl} +1$.  In $\Ssz$, we will have tuples of the form
$(\star,m_{ij}-(m_{ij}-m_{kl}+1),\star,m_{il} - (m_{ij}-m_{kl}+1),\star)$
and $(\star,m_{kj},\star,0,\star)$. But then $\Ssz$
will have incomparable elements
since $m_{ij}-(m_{ij}-m_{kl}+1) <
m_{kj}$ but $m_{il} - (m_{ij}-m_{kl}+1) > 0$.
\end{proof}

In the sequel we will require a bigraded version of  Bezout's theorem:

\begin{theorem} \label{bezout}
Let $F,G \in k[x_0,x_1,y_1,y_0]$ be two bihomogeneous forms such
that $G$ is irreducible, $\deg G = (a,b)$, and $\deg F =
(c,d)$.  If the curves defined by $F$ and $G$ in $\popo$ meet at
more than $ad + bc$ points (counting multiplicities), then $F =
GF'$.
\end{theorem}

\subsection{Separators of fat points} In \cite{GVT3},
the authors introduced the notion of
a separator for a set of fat points in $\pnpm$.  We recall
these results, but specialize to the case of $\popo$.  
Note, we will denote
a point of $\popo$ simply by
$P$ instead $P_i \times Q_j$.

\begin{definition}
Let $Z = m_1P_1 + \cdots + m_iP_i + \cdots
+ m_sP_s$ be a set of fat points in $\popo$.  We say
that $F$ is a {\bf separator of the point $P_i$ of multiplicity
$m_i$} if $F \in I_{P_i}^{m_i-1} \setminus I_{P_i}^{m_i}$
and $F \in I_{P_j}^{m_j}$ for all $j \neq i$.
\end{definition}

If we let $Z' = m_1P_1 + \cdots + (m_i-1)P_i + \cdots
+ m_sP_s$, then a separator of the point $P_i$ of multiplicity
$m_i$ is also an element of $F \in I_{Z'}\setminus I_Z$.  The
set of minimal separators are defined in terms of the ideals
$I_{Z'}$ and $I_Z$.

\begin{definition}
A set $\{F_1,\ldots,F_p\}$ is a set of {\bf minimal separators of $P_i$
of multiplicity $m_i$} if $I_{Z'}/I_Z = (\overline{F}_1,\ldots,\overline{F}_p)$,
and there does not exist a set $\{G_1,\ldots,G_q\}$ with
$q < p$ such that  $I_{Z'}/I_Z = (\overline{G}_1,\ldots,\overline{G}_q)$.
\end{definition}

Important for this paper is the following definition:

\begin{definition}
The {\bf degree of the minimal separators of $P_i$ of multiplicity
$m_i$}, denoted $\deg_Z(P_i)$, is the tuple
\[\deg_Z(P_i) = (\deg F_1,\ldots,\deg F_p) ~~ \text{where $\deg F_i \in \N^2$}\]
and $F_1,\ldots,F_p$ is any set of minimal separators of $P_i$ of multiplicity
$m_i$.
\end{definition}

For a general fat point scheme $Z \subseteq \popo$, there is no
known formula for $p = |\deg_Z(P)|$.  However, as a special case
of Theorem {4.3} and Theorem {5.1} of \cite{GVT3}, we can compute
the exact value for $p$ if we also assume that $Z$ is ACM (as we
shall assume throughout the next section):

\begin{theorem} \label{degpoint} Let $Z \subseteq \popo$ be an ACM set of fat points.
If $P$ is a fat point of multiplicity $m$ of $Z$,
then $|\deg_Z(P)| = m$.
\end{theorem}

\section{Main Results}

\subsection{Fat points on a ruling}

We begin by looking at a special case, namely, $\supp(Z) = \{P
\times Q_1,\ldots, P \times Q_b\}$, i.e., all the points have the
same first coordinate.   The fact that these schemes are ACM
follows directly from Theorem \ref{equivalent}.  We first require
a lemma which depends upon the Hilbert functions of these 
schemes (see \cite[Theorem 2.2]{GV1}).

\begin{lemma}\label{hilbertptsonaline}
Let $Z$ be a set of fat points in $\popo$ of the form
\[Z = m_1(P \times Q_1) + m_2(P \times Q_2) + \cdots + m_b(P \times Q_b).\] 
Let $m = \max\{m_j\}_{j=1}^b$.  For $\ell = 0,\ldots,m-1$,
set $c_{\ell} = \sum_{p=1}^b (m_p - \ell)_+$.  If $(i,j) \not\succeq (\ell,c_{\ell})$
for all $\ell \in \{0,\ldots,m-1\}$, then $\dim_k (I_Z)_{i,j} = 0$.
\end{lemma}

\begin{theorem}\label{ptsonaline1}
Let $Z$ be a set of fat points in $\popo$ of the form
\[Z = m_1(P \times Q_1) + m_2(P \times Q_2) + \cdots + m_i(P \times Q_i)
+ \cdots + m_b(P \times Q_b),\] i.e., each point of $\supp(Z)$ has
the same first coordinate.  Fix an $i \in \{1,\ldots,b\}$, and set
\[b_{\ell} = \sum_{p=1}^b (m_{p} - \ell)_+
~~\mbox{for $\ell = 0,\ldots,m_{i} -1$.}\] Then
$\deg_Z(P \times Q_i) =   \{(\ell,b_{\ell}-1) ~|~ \ell=0,\ldots,m_{{i}-1}\}.$
\end{theorem}

\begin{proof}
By Theorem
\ref{degpoint}, we have $|\deg_{Z}(P \times Q_i)| = m_i.$
We first
construct $m_{i}$ separators $F_0,\ldots,F_{m_{i}-1}$ of $P \times
Q_i$ of multiplicity $m_{i}$ where $\deg(F_\ell) =
(\ell,b_\ell-1)$ for $\ell = 0,\ldots,m_{i}-1$. Our second step is
to prove that these separators form a set of minimal separators.

To simplify our notation,  let $R  = L_P$ and $L_i = L_{Q_i}$ for
$i=1,\ldots,b$.  Fix an $\ell \in
\{0,\ldots,m_{i}-1\}$, and let \[A_{\ell}=R^{\ell} \text{ and }
B_{\ell} = L_1^{(m_{1}-\ell)_+} L_2^{(m_{2}-\ell)_+} \cdots
 L_i^{(m_{i}-\ell)_+-1} \cdots
 L_b^{(m_{b}-\ell)_+}\]
We then set $F_{\ell} = A_{\ell}B_{\ell}$.  By construction,
$\deg(F_{\ell}) = (\ell,b_{\ell}-1)$.

We now show that $F_{\ell} \in I_{Z'} \setminus I_Z$ where
$Z'$ denotes the set of fat points with the multiplicity of $P
\times Q_i$ reduced by one. Note that $F_{\ell} = R^{\ell}
L_i^{(m_{i}-\ell)_+-1} F'_{\ell}$, and since $\ell + (m_{i}-\ell)_+-1 = m_{i} -1,$
we have $F_{\ell} \not\in I_{Z}$ since $F _{\ell}\in I_{P\times
Q_i}^{m_{i}-1} \setminus I_{P \times Q_i}^{m_{i}}$.  This follows from 
the fact that $F'_{\ell}$ does not pass through
$P \times Q_i$.  Now let $P \times Q_f$ be any other point
in $\supp(Z)$  distinct from $P \times Q_i$.  Since $I_{P \times Q_f} =
(R,L_f)^{m_{f}}$, and since the exponents of $R$ and $L_f$ in
$F_{\ell}$ sum up to at least $m_{f}$, we have that $F _{\ell}\in
I_{P\times Q_f}^{m_{f}}$.  So $F_{\ell} \in I_{Z'} \setminus I_Z$.

Let us now show that the $F_\ell$ are minimal separators.  Let
$G$ be any separator of $P \times Q_i$ of multiplicity
$m_{i}$ with $\deg(G) = (c,d)$.   We want to show
that $\deg(G) \succeq (\ell,b_{\ell}-1)$ for some $\ell \in
\{0,\ldots,m_i-1\}$. If we can verify this fact, then the explicit
separators described above would form a minimal set of minimal
separators.

Suppose that for every $\ell$, $(c,d) \not\succeq
({\ell},b_{\ell}-1)$.  Since $Z'$ is also a fat point
scheme of points on a line, by Lemma \ref{hilbertptsonaline}
$\dim_k(I_{Z'})_{c,d} = 0$  ($b_{\ell}-1$ appears as some $c_{\ell}$
because we have reduced the multiplicty of $m_i$ by $1$).
So, $0 \neq G \in (I_{Z'})_{c,d} = (0)$, a contradiction.
\end{proof}

\begin{remark}\label{ptsonaline2}  By swapping the roles of the grading,
we can prove a similar result for sets of points whose second coordinate
are the same.  We leave it to the reader to write out the corresponding
statement of Theorem \ref{ptsonaline1}.
\end{remark}

\subsection{Separators of ACM fat points}
The main result of this paper is a formula to compute the
degree of a minimal separator for each fat point in
an ACM fat point scheme in $\popo$.

\begin{theorem}\label{degP1xP1}
Let $Z \subseteq \popo$ be an ACM set of fat points.  For any $P_i
\times Q_j \in \supp(Z)$, let
\[Y =m_{1j}(P_1 \times Q_j) +
\cdots + m_{ij}(P_i \times Q_j) + \cdots + m_{aj}(P_a \times
Q_j)\} \] be all the fat points of $Z$ whose support has  $P_i$ as
its first coordinate, and let
\[W = m_{i1}(P_i \times Q_1) + \cdots + m_{ij}(P_i \times Q_j)
+ \cdots + m_{ib}(P_i \times Q_b)\}\] be all the fat points of $Z$
whose support has  $Q_j$ as its second coordinate.  Set
\[a_{\ell} = \sum_{s=1}^a (m_{sj} - \ell)_+ ~~\text{and}
~~b_{\ell} = \sum_{p=1}^b (m_{ip} - \ell)_+ ~\hspace{.25cm}
\mbox{for $\ell = 0,\ldots,m_{ij}-1$.}\] Then
\[
\deg_Z(P_i \times Q_j) =  \{(a_{m_{ij}-1-\ell}-1,b_{\ell}-1) ~|~ \ell =
0,\ldots,m_{ij}-1\}.\]
\end{theorem}

Before proving this result, let us illustrate how to use it.

\begin{example} We continue to use the example
of Example \ref{mainexample}.  For convenience, we recall that
\begin{center}
\begin{picture}(150,110)(25,-10)
\put(0,40){$Z = $} \put(60,-10){\line(0,1){90}}
\put(80,-10){\line(0,1){90}} \put(100,-10){\line(0,1){90}}
\put(120,-10){\line(0,1){90}} \put(58,95){$Q_1$}
\put(78,95){$Q_2$} \put(98,95){$Q_3$} \put(118,95){$Q_4$}

\put(64,77){$5$} \put(84,77){$4$} \put(104,77){$2$}
\put(124,77){$2$}

\put(64,57){$5$} \put(84,57){$3$} \put(104,57){$2$}
\put(124,57){$1$}

\put(64,37){$4$} \put(84,37){$3$} \put(104,37){$1$}

\put(64,17){$2$} \put(84,17){$1$}

\put(64,-2){$1$}

\put(55,-5){\line(1,0){70}} \put(55,15){\line(1,0){70}}
\put(55,35){\line(1,0){70}} \put(55,55){\line(1,0){70}}
\put(55,75){\line(1,0){70}} \put(37,73){$P_1$} \put(37,53){$P_2$}
\put(37,33){$P_3$} \put(37,13){$P_4$} \put(37,-7){$P_5$}

\put(60,75){\circle*{5}} \put(80,75){\circle*{5}}
\put(100,75){\circle*{5}} \put(120,75){\circle*{5}}

\put(60,55){\circle*{5}} \put(80,55){\circle*{5}}
\put(100,55){\circle*{5}} \put(120,55){\circle*{5}}

\put(60,35){\circle*{5}} \put(80,35){\circle*{5}}
\put(100,35){\circle*{5}}

\put(60,15){\circle*{5}} \put(80,15){\circle*{5}}
\put(60,-5){\circle*{5}}
\end{picture}
\end{center}
We will compute $\deg_Z(P_3 \times Q_2)$.  The multiplicity of
$P_3 \times Q_2$ is $m_{3,2} = 3$, so we will have $|\deg_Z(P_3
\times Q_2)| = 3$. In the notation of Theorem \ref{degP1xP1}, we
have
\[Y =4(P_1 \times Q_2) + 3(P_2 \times Q_2) + 3(P_3 \times Q_2) + 1(P_4 \times Q_2)\]
and
\[W = 4(P_3 \times Q_1) + 3(P_3 \times Q_2) + 1(P_3 \times Q_3).\]
We now calculate $a_0,a_1,a_2$, and $b_0,b_1,b_2$:
\[\begin{array}{lcrclclr}
a_0 & = & 4+3+3+1 &\hspace{.25cm} & b_0 & = & 4+3+1 \\
a_1 & = & 3+2+2+0 && b_1 & = & 3+2+0 \\
a_2 & = & 2+1+1+0 && b_2 & = & 2+1+0. \\
\end{array}\]
We thus get
\[
\deg_Z(P_3 \times Q_2) = \{(a_0-1,b_2-1),(a_1-1,b_1-1),(a_2-1,b_0-1)\}
=  \{(10,2),(6,4),(3,7)\}.\]
Furthermore, as we will describe in the proof of Theorem
\ref{degP1xP1}, we can explicitly determine these minimal
separators. If $L_{P_i}$ denotes the degree $(1,0)$ form that
passes through $P_i$ and $L_{Q_j}$ denotes the degree $(0,1)$ form
that passes through $Q_j$, then the forms
$F_1=L_{P_1}^4L_{P_2}^3L_{P_3}^2L_{P_4}L_{Q_1}^2,$
$F_2=L_{P_1}^3L_{P_2}^2L_{P_3}L_{Q_1}^3L_{Q_2}$ and
$F_3=L_{P_1}^2L_{P_2}L_{Q_1}^4L_{Q_2}^2L_{Q_3}$ are the minimal
separators of $P_3 \times Q_2$ of multiplicity $m_{3,2}=3$ with the
required degrees.
\end{example}

\begin{proof} (of Theorem \ref{degP1xP1})
Fix a point $P_i \times Q_j$ of multiplicity $m_{ij}$ in $Z$.
There are two main steps.  First, we
construct $m_{ij}$ separators $F_0,\ldots,F_{m_{ij}-1}$ of $P_i
\times Q_j$ of multiplicity $m_{ij}$ where $\deg(F_\ell) =
(a_{m_{ij}-\ell-1}-1,b_\ell-1)$ for $\ell = 0,\ldots,m_{ij}-1$.
Second, we prove that these separators form a set of
minimal separators of $P_i \times Q_j$ of multiplicity $m_{ij}$.

To simplify our notation slightly, let $L_i$ denote the degree
$(1,0)$ for that passes through $P_i$ for $i=1,\ldots,a$,
and let $R_j$ denote the
degree $(0,1)$ form that passes through $Q_j$ for $j=1,\ldots,b$.  Fix an $\ell \in
\{0,\ldots,m_{ij}-1\}$, and let
\[A_{\ell}= L_1^{(m_{1j}-(m_{ij}-\ell-1))_+}L_2^{(m_{2j}-(m_{ij}-\ell-1))_+}
\cdots L_i^{(m_{ij}-(m_{ij}-\ell-1))_+-1} \cdots
L_a^{(m_{aj}-(m_{ij}-\ell-1))_+}\] and
\[B_{\ell} = R_1^{(m_{i1}-\ell)_+} R_2^{(m_{i2}-\ell)_+} \cdots
 R_j^{(m_{ij}-\ell)_+-1} \cdots
 R_b^{(m_{ib}-\ell)_+}\]
We then set $F = F_{\ell} = A_{\ell}B_{\ell}$.  By construction,
$\deg(F) = (a_{m_{ij}-\ell-1}-1,b_{\ell}-1)$.

We now need to show that $F \in I_{Z'} \setminus I_Z$ where $Z'$
denotes the set of fat points with the multiplicity of $P_i \times
Q_j$ reduced by one. Note that $F =
L_i^{(m_{ij}-(m_{ij}-\ell-1))_+-1}  R_j^{(m_{ij}-\ell)_+-1} F'$,
and since \[(m_{ij}-(m_{ij}-\ell-1))_+-1 + (m_{ij}-\ell)_+-1 =
m_{ij} -1,\] we have $F \not\in I_{Z}$ since $F \in I_{P_i\times
Q_j}^{m_{ij}-1} \setminus I_{P_i \times Q_j}^{m_{ij}}$.
This is because $F'$ does not pass through $P_i
\times Q_j$.

Now take any other point $P_e \times Q_f$ in the support of $Z$
distinct from $P_i \times Q_j$.  We need to show that $F \in
I_{P_e \times Q_f}^{m_{ef}}$.  Since $I_{P_e \times Q_f} =
(L_e,R_f)^{m_{ef}} = (L_e^uR_f^v ~|~ u+v = m_{ef})$, it will suffice to show that the exponents of
$L_e$ and $R_f$ in $F$ sum up to at least $m_{ef}$. We break this
problem into a number of cases.  Recall that since we are assuming
that $Z$ is ACM, we can assume that $m_{ab} \geq m_{cb}$ if $a <
c$ and $m_{ab} \geq m_{ad}$ if $b < d$.

\noindent {\bf Case 1}:  $e <i$ and $f < j$.

We have $F =
L_e^{(m_{ej}-(m_{ij}-\ell-1))_+}R_f^{(m_{if}-\ell)_+}F'$. We have
\[(m_{ej}-(m_{ij}-\ell-1))_+ + (m_{if}-\ell)_+ = m_{ej}+m_{if}-m_{ij}+1\]
because $m_{ej} \geq m_{ij}$ and $m_{if} - \ell \geq m_{ij} - \ell
\geq 1$. But by Lemma \ref{bounds}, we have
$m_{ej}+m_{if}-m_{ij}+1 \geq m_{ef}$.

\noindent {\bf Case 2}: $e < i$ and $f =j$.

We observe that $F =
L_e^{(m_{ej}-(m_{ij}-\ell-1))_+}R_j^{(m_{ij}-\ell)_+-1}F'$. So, we
have
\[(m_{ej}-(m_{ij}-\ell-1))_+ + (m_{ij}-\ell)_+-1 =
m_{ej}-m_{ij}+\ell+1+m_{ij}-\ell-1 = m_{ej}\] since $m_{ej} \geq
m_{ij}$ and $m_{ij} - \ell \geq 1$.

\noindent {\bf Case 3}: $e < i$ and $j < f \leq b$.

In this case, $F =
L_e^{(m_{ej}-(m_{ij}-\ell-1))_+}R_f^{(m_{if}-\ell)_+}F'$. There
are two possibilities: (a) $m_{if} - \ell \geq 0$, and (b) $m_{if}
- \ell < 0$.  If $(a)$ holds, then since $m_{ej} \geq m_{ij}$
\[(m_{ej}-(m_{ij}-\ell-1))_+ + (m_{if}-\ell)_+ = m_{ej}+m_{if}-m_{ij} +1 \geq m_{ef}\]
where the last inequality holds by Lemma \ref{bounds}. If $(b)$
holds, then $R_f$ does not appear as a factor of $F$. But
$(m_{ej}-(m_{ij}-\ell-1))_+ = m_{ej}-m_{ij}+ \ell +1 >
m_{ej}-m_{ij}+m_{if}+1 \geq m_{ef}$, where the last inequality
again holds by Lemma \ref{bounds}. \noindent

\noindent {\bf Case 4}: $e < i$ and $f > b$.

In this case, $R_f$ is not a factor of $F$, and so $F =
L_e^{(m_{ej}-(m_{ij}-\ell-1))_+}F'$.  Note that since $f >b$,
the points $P_{ej},P_{ef}$, and $P_{ij}$ are in the
support of $Z$, but $P_{if}$ is not. By Lemma \ref{bounds},  we
have $m_{ej} - m_{ij}+1 \geq m_{ef}$. So
$(m_{ej}-(m_{ij}-\ell-1))_+ \geq m_{ef}$.

\noindent {\bf Case 5}: $e=i$.

If $e = i$, then $f \in \{1,\ldots,\hat{j},\ldots,b\}$, and
furthermore, $F =
L_i^{(m_{ij}-(m_{ij}-\ell-1))_+-1}R_f^{(m_{if}-\ell)_+}F'$. If
$m_{if}-\ell \geq 0$, then $(m_{ij}-(m_{ij}-\ell-1))_+-1 +
(m_{if}-\ell)_+ = m_{if}$.  If $m_{if} - \ell < 0$, then
$(m_{ij}-(m_{ij}-\ell-1))_+-1 + (m_{if}-\ell)_+ = \ell > m_{if}.$

\noindent {\bf Case 6}: $i < e \leq a$ and $f < j$.

This case is similar to Case 3.

\noindent {\bf Case 7}: $i < e \leq a$ and $f =j$.

In this situation, $F =
L_e^{(m_{ej}-(m_{ij}-\ell-1))_+}R_j^{(m_{ij}-\ell)_+-1}F'$. There
are two possibilities.  If $(m_{ej}-(m_{ij}-\ell-1)) < 0$, then
the sum of the exponents is simply $m_{ij}-\ell -1 > m_{ej}$. On
the other hand, if $(m_{ej}-(m_{ij}-\ell-1)) \geq 0$, the sum of the
exponents is $(m_{ej}-(m_{ij}-\ell-1))_+ + (m_{ij}-\ell)_+-1 =
m_{ej}$, as desired.

\noindent {\bf Case 8}: $i < e \leq a$ and $f > j$.

We have $F =
L_e^{(m_{ej}-(m_{ij}-\ell-1))_+}R_f^{(m_{if}-\ell)_+}F'$. We need
to consider four possibilities:
\begin{enumerate}
\item[$(a)$] $(m_{ej}-(m_{ij}-\ell-1)) \geq 0$
and $(m_{if}-\ell) \geq 0$.  In this case, the exponents sum to
$m_{ej}+m_{if}-m_{ij} +1 \geq m_{ef}$ by Lemma \ref{bounds};
\item[$(b)$] $(m_{ej}-(m_{ij}-\ell-1)) \geq 0$
and $(m_{if}-\ell) < 0$.  In this case, the exponents sum to
$m_{ej} - m_{ij} + \ell+1 > m_{ej} +m_{if}-m_{ij}+1 \geq m_{ef}$
by Lemma \ref{bounds};
\item[$(c)$]  $(m_{ej}-(m_{ij}-\ell-1)) < 0$
and $(m_{if}-\ell) \geq 0$.  We know that $m_{ef} \leq
m_{ej}+m_{if}-m_{ij}+1$. But $m_{ej}-m_{ij} < - \ell -1$.  So
$m_{ef} < m_{if} - \ell+1-1 = m_{if}- \ell$, as desired.
\item[$(d)$] $(m_{ej}-(m_{ij}-\ell-1)) < 0$
and $(m_{if}-\ell) < 0$.  This case cannot occur.  If it did, then
we would have $0 > (m_{ej}-(m_{ij}-\ell-1)) + (m_{if}-\ell) =
m_{ej}+m_{if}-m_{ij}+1 \geq m_{ef} \geq 0$.
\end{enumerate}

\noindent {\bf Case 9}: $a > e$ and $f <j$.

This case is the same as Case 4.

These nine cases now show that each $F_{\ell}$ with $\ell \in \{0,\ldots,m_i-1\}$ 
is a separator of
$P_i \times Q_j$ of multiplicity $m_{ij}$.  We now demonstrate
that these are the minimal separators.

Let $F$ be any separator of $P_i \times Q_j$ of multiplicity
$m_{ij}$ with $\deg(F) = (c,d)$.  To simplify our notation, set
$m= m_{ij}$. We want to show that $\deg(F) \succeq
(a_{m-1-\ell}-1,b_{\ell}-1)$ for some $\ell \in \{0,\ldots,m-1\}$.
If we can verify this fact, then the explicit separators described
above would form a set of minimal separators of $P_i \times Q_j$
of multiplicity $m$.

So, suppose that for every $\ell$, $(c,d) \not\succeq
(a_{m-1-\ell}-1,b_{\ell}-1)$. Note, however, that $F$ is a
separator of $P_i \times Q_j$ of multiplicity of $m$ of $Y$.  So,
there exists a $t$ such that $(c,d) \succeq (a_{m-1-t}-1,m-1-t)$
by Remark \ref{ptsonaline2} and Theorem \ref{ptsonaline1}.
However, since $(c,d) \not\succeq
(a_{m-1-t}-1,b_t-1)$, this implies that $d < b_t$.  On the other
hand, by Theorem \ref{ptsonaline1}, there exists a $k$ such that
$(c,d) \succeq (k,b_k-1)$ since $F$ is also a separator of $P_i
\times Q_j$ of multiplicity of $m$ of $W$.  We thus have
that $b_k-1 \leq d < b_t -1$, whence $t < k$.

Since $b_k < b_{k-1} < \cdots < b_t$, there exists a $p$ such that
$b_{p+1}-1 \leq d < b_p-1$ with $t \leq p <k$. Now $(c,d) \succeq
(k,b_k-1)$, and $k \geq p+1$ and $d \geq b_{p+1}-1$, so we also
have $(c,d) \succeq (p+1,b_{p+1}-1)$.  However, because $(c,d)
\not\succeq (a_{m-1-(p+1)}-1,b_{p+1}-1)$, we have $c < a_{m-p-2}
-1$.

Consider the line $L_i$, i.e., the degree $(1,0)$ form that
contains $W$.  Then the curves defined by $F$ and $L_i$ meet at
$m_{i1}+m_{i2}+\cdots + (m_{ij}-1) + \cdots + m_{ib} = b_0 -1$
points. By Bezout's Theorem (see Theorem \ref{bezout}), since
$L_i$ is irreducible, and $c\cdot 0 + d\cdot 1 = d < b_p-1 < b_0
-1$, we have that $F = F_0L_i$.  But now consider $F_0$.  This is
a form of degree $(c-1,d)$, and the curve it defines meets $L_i$
at
\[(m_{i1}-1)_++(m_{i2}-1)_++\cdots + (m_{ij}-1-1)_+ + \cdots + (m_{ib}-1)_+ \geq b_1 -1\]
points (the inequality comes from the fact that $(m_{ij}-1-1)_+
\geq [(m_{ij}-1)_+-1$]). By Bezout's Theorem $F_0 = F_1L_i$ because
$(c-1)\cdot 0 + d\cdot 1 = d < b_p-1 < b_1 -1$.  We can continue
this argument until will arrive at $F = F_pL_i^{p+1}$, where
$\deg(F_p) = (c-p-1,d)$.

The form $F_p$ and the degree $(0,1)$ form $R_j$ which contains
$Y$ meet at
\[m_{1j} + m_{2j} + \cdots + (m_{ij}-p-2)_+ +\cdots + m_{aj} \geq  a_0 - p -2\]
points, counting multiplicities.  To see this, note that
$L_i^{p+1}$ already passes through the point $P_i \times Q_j$
$(p+1)$ times. Since $F$ passes through $P_i \times Q_j$ with
multiplicity $m -1$, $F_p$ must pass through $P_i \times Q_j$
$(m-p-2)_+$ times. The inequality comes from the fact that
$(m_{ij}-p-2)_+ \geq m_{ij}-p -2$.

Now since $c < a_{m-p-2}- 1$, we have $c-p-1 < a_{m-p-2}-p-2 <
\cdots < a_0 - p-2$.  So, by Bezout's Theorem (Theorem
\ref{bezout}), we get $F_p = G_0R_j$.  But then $G_0$ has degree
$(c-p-1,d-1)$, and we can repeat the above argument to show that
$G_0 = G_1R_j$.  Continuing in this fashion, we arrive at $F_p =
G_{m-p-2}R_j^{m-p-1}$.

We therefore have $F = L_i^{p+1}R_j^{m-p-1}G_{m-p-2}$ where
$G_{m-p-2}$ has degree $(c-p-1,d-m+p+1)$. The exponents of $L_i$
and $R_j$ sum to $m$, which means that $F \in (L_i,R_j)^{m}$,
which contradicts the fact that a separator of $P_i \times Q_j$ of
multiplicity $m = m_{ij}$ belongs to $(L_i,R_j)^{m-1} \setminus
(L_i,R_j)^{m}$.
Thus, there cannot be a separator $F$ with degree $(c,d)
\not\succeq (a_{m-1-\ell}-1,b_{\ell}-1)$ for all $\ell$.
\end{proof}

\begin{application}  We sketch out how one might use
Theorem \ref{degP1xP1} to compute some Hilbert functions.
When $Z$ is an ACM fat point scheme in $\popo$,
then $H_Z(i,j)$ can be computed for all $(i,j)$ directly
from the set $\Ssz$ introduced in Section 2 (see \cite{Gu2} for
complete details).  If we pick any fat point $P_i \times Q_j$
in $Z$ of multiplicity $m_{ij}$, then by Theorem \ref{degP1xP1},
we can compute
\[\deg_Z(P_i\times Q_j) = ((c_1,d_1),\ldots,(c_{m_{ij}},d_{m_{ij}})).\]
Let $Z'$ be the scheme formed by reducing the multiplicity of $P_i
\times Q_j$ by one.  As shown in \cite[Corollary 4.4]{GVT3}, we
can compute the Hilbert function of $Z'$ as follows:
\[
H_{Z'}(r,s) =  H_Z(r,s) - |\{(c,d)\in \deg(P_i\times Q_j)|
(c,d)\preceq(r,s)\}|.
\]
In other words, if $Z'$ is any fat point scheme (possibly
not ACM) which has the
property that if we increase the multiplicity of one of its points
by one to get an ACM scheme, then the Hilbert
function of $Z'$ can be computed directly from numerical
information describing $Z'$.
\end{application}


\end{document}